\documentclass[12pt,oneside,english]{amsart}
\usepackage[T1]{fontenc}
\usepackage[latin9]{inputenc}
\usepackage{geometry}
\usepackage{mathtools}
\usepackage{amstext}
\usepackage{amsthm}
\usepackage{amssymb}

\makeatletter
\numberwithin{equation}{section}
\numberwithin{figure}{section}
\theoremstyle{plain}
\newtheorem{thm}{\protect\theoremname}
  \theoremstyle{definition}
  \newtheorem{defn}[thm]{\protect\definitionname}
  \theoremstyle{plain}
  \newtheorem{lem}[thm]{\protect\lemmaname}
  \theoremstyle{definition}
  \newtheorem{example}[thm]{\protect\examplename}
  \theoremstyle{plain}
  \newtheorem{prop}[thm]{\protect\propositionname}
  \theoremstyle{definition}
  \newtheorem{problem}[thm]{\protect\problemname}
  \theoremstyle{remark}
  \newtheorem{rem}[thm]{\protect\remarkname}

\DeclareMathOperator{\sgn}{sgn}
\DeclareMathOperator{\slicerank}{slice-rank}
\DeclareMathOperator{\prank}{partition-rank}
\DeclareMathOperator{\trank}{tensor-rank}
\DeclareMathOperator{\dashrank}{-rank}
\DeclareMathOperator{\spann}{span}
\DeclareMathOperator{\Poly}{Poly}

\makeatother

\usepackage{babel}
  \providecommand{\definitionname}{Definition}
  \providecommand{\examplename}{Example}
  \providecommand{\lemmaname}{Lemma}
  \providecommand{\problemname}{Problem}
  \providecommand{\propositionname}{Proposition}
  \providecommand{\remarkname}{Remark}
\providecommand{\theoremname}{Theorem}

\usepackage{hyperref}
\PassOptionsToPackage{hyphens}{url}\usepackage{hyperref}

\providecommand{\MR}{\relax\ifhmode\unskip\space\fi MR }
\newcommand\arXiv[1]{arXiv:\href{http://arXiv.org/abs/#1}{#1}}

\begin{document}

\title{The Partition Rank of a Tensor and $k$-Right Corners in $\mathbb{F}_{q}^{n}$}

\author{Eric Naslund}

\date{\today}
\begin{abstract}
Following the breakthrough of Croot, Lev, and Pach \cite{CrootLevPachZ4},
Tao \cite{TaosBlogCapsets} introduced a symmetrized version of their
argument, which is now known as the slice rank method. In this paper,
we introduce a more general version of the slice rank of a tensor,
which we call the \emph{Partition Rank.} This allows us to extend
the slice rank method to problems that require the variables to be
distinct. Using the partition rank, we generalize a recent result
of Ge and Shangguan \cite{GeShangguan2017NoRightAngles}, and prove
that any set $A\subset\mathbb{F}_{q}^{n}$ of size 
\[
|A|>\binom{n+(k-1)q}{(k-1)(q-1)}
\]
contains a \emph{$k$-right-corner}, that is distinct vectors $x_{1},\dots,x_{k},x_{k+1}$
where $x_{1}-x_{k+1},\dots,x_{k}-x_{k+1}$ are mutually orthogonal,
for $q=p^{r}$, a prime power with $p>k$.
\end{abstract}

\maketitle

\section{Introduction}

In the spring of 2016, Croot, Lev and Pach \cite{CrootLevPachZ4}
proved a breakthrough result on progression-free subsets of $\left(\mathbb{Z}/4\mathbb{Z}\right)^{n}$.
They proved that if $A\subset\left(\mathbb{Z}/4\mathbb{Z}\right)^{n}$
contains no non-trivial three term arithmetic progression, then 
\[
|A|\leq3.60172^{n}.
\]
Ellenberg and Gijswijt \cite{EllenbergGijswijtCapsets} used Croot,
Lev, and Pach's method to prove that any progression-free subset $A\subset\left(\mathbb{Z}/p\mathbb{Z}\right)^{n}$,
where $p$ is a prime, satisfies 
\[
|A|\leq\left(J(p)p\right)^{n},
\]
where 
\[
J(p)=\frac{1}{p}\min_{0<t<1}\frac{1-t^{p}}{(1-t)t^{\frac{p-1}{3}}}
\]
is an explicit constant less than $1$. This was extended to $\left(\mathbb{Z}/k\mathbb{Z}\right)^{n}$
for any integer $k$ in \cite{BlasiakChurchCohnGrochowNaslundSawinUmans2016MatrixMultiplication}.
Following the work of Ellenberg and Gijswijt, Tao, in his blog \cite{TaosBlogCapsets},
symmetrized Croot, Lev, and Pach's argument, introducing the notion
of the \emph{slice rank }of a tensor\emph{. }The slice rank method
has seen numerous applications to a variety of problems, such as the
sunflower problem \cite{NaslundSawinSunflower}, right angles in $\mathbb{F}_{p}^{n}$
\cite{GeShangguan2017NoRightAngles}, and approaches to fast matrix
multiplication \cite{BlasiakChurchCohnGrochowNaslundSawinUmans2016MatrixMultiplication},
and we refer the reader to \cite[Section 4]{BlasiakChurchCohnGrochowNaslundSawinUmans2016MatrixMultiplication}
for an in depth discussion of the slice rank as well as its connection
to geometric invariant theory. 

In this paper, we introduce the \emph{partition rank} of a tensor,
which generalizes the slice rank, and allows us to handle problems
that require the variables to be distinct. Using this new notion of
rank, we generalize the work of Ge and Shangguan \cite{GeShangguan2017NoRightAngles}
to $k$-right corners in $\mathbb{F}_{q}^{n}$. In a different direction,
in \cite{Naslund2017EGZ}, we use the partition rank to provide an
exponential improvement to bounds for the Erd\"{o}s-Ginzburg-Ziv
constant of $\left(\mathbb{Z}/p\mathbb{Z}\right)^{n}$, which are
the first non-trivial upper bounds for high rank abelian groups.

Let $q$ be an odd prime power. A \emph{right angle} in $\mathbb{F}_{q}^{n}$
is a triple $x,y,z\in\mathbb{F}_{q}^{n}$ of distinct elements satisfying
\[
\langle x-y,x-z\rangle=0.
\]
Bennett \cite{Bennett2015RightAngles} proved that any subset $A\subset\mathbb{F}_{q}^{n}$
of size 
\[
|A|\gg q^{\frac{n+2}{3}}
\]
contains a right angle. Ge and Shangguan \cite{GeShangguan2017NoRightAngles}
used the slice rank method to improve this result for very large $n$,
showing that any set $A\subset\mathbb{F}_{q}^{n}$ of size 
\[
|A|>\binom{n+q}{q-1}+3
\]
contains a right angle. In section \ref{sec:Right-Angles}, we improve
Ge and Shangguan's bound, and prove the following theorem:
\begin{thm}
\label{thm:right-angle}Let $q$ be an odd prime power. If $A\subset\mathbb{F}_{q}^{n}$
satisfies
\begin{equation}
|A|>\binom{n+q}{q-1}+2-\binom{n+q-2}{q-3}\label{eq:right_angle_lower_bound}
\end{equation}
then $|A|$ contains a right angle.
\end{thm}
Alternatively, we may write the right hand side of (\ref{eq:right_angle_lower_bound})
as
\[
\binom{n+q-1}{q-1}+\binom{n+q-2}{q-2}+2.
\]

Generalizing the notion of a right angle in $\mathbb{F}_{q}^{n}$,
we say that the vectors $x_{1},\dots,x_{k},x_{k+1}$ form a $k$\emph{-right
corner} if they are distinct, and if the $k$ vectors $x_{1}-x_{k+1},\dots,x_{k}-x_{k+1}$
form a mutually orthogonal $k$-tuple. In other words, $x_{1},\dots,x_{k}$
must meet $x_{k+1}$ at mutually right angles. Our main result is
a bound, polynomial in $n$, for the size of the largest subset of
$\mathbb{F}_{q}^{n}$ that does not contain a $k$-right corner.
\begin{thm}
\label{thm:Right-Corner} Let $k$ be given, and let $q=p^{r}$ with
$p>k$. If $A\subset\mathbb{F}_{q}^{n}$ satisfies
\[
|A|>\binom{n+(k-1)q}{(k-1)(q-1)},
\]
then $A$ contains a $k$-right corner.
\end{thm}
The proof of theorem \ref{thm:Right-Corner} relies on the flexibility
of the partition rank for $k\geq3$, as the tensors involved will
have polynomially small partition rank yet exponentially large slice
rank.

In subsection \ref{subsec:The-Slice-Rank}, we review the slice rank,
as defined in \cite{TaosBlogCapsets} and \cite[Section 4.1]{BlasiakChurchCohnGrochowNaslundSawinUmans2016MatrixMultiplication},
and in subsection \ref{subsec:The-Partition-Rank} we introduce the
partition rank and prove the critical lemma, which states that the
partition rank of a diagonal tensor is equal to the number of non-zero
diagonal entries. In section \ref{sec:The-Distinctness-Indicator},
we introduce the indicator function $H_{k}(x_{1},\dots,x_{k})$, which
satisfies 
\[
H_{k}(x_{1},\dots,x_{k})=\begin{cases}
1 & \text{if }x_{1},\dots,x_{k}\text{ are distinct},\\
(-1)^{k-1}(k-1)! & \text{if }x_{1}=\cdots=x_{k},\\
0 & \text{otherwise}.
\end{cases}
\]
For $k\geq4$, $H_{k}$ will have low partition rank, but large slice
rank. This function allows us to modify our tensor so that it picks
up only $k$-tuples of distinct elements. In section \ref{sec:Right-Corners},
we use the partition rank and the indicator function $H_{k}$ to prove
theorem \ref{thm:Right-Corner}.

\section{The Slice Rank and the Partition Rank\label{sec:The-Partition-Rank}}

\subsection{\label{subsec:The-Slice-Rank}The Slice Rank}

We begin by recalling the definition of the rank of a two variable
function. Let $X,Y$ be finite sets, and suppose that $\mathbb{F}$
is a field. The rank of 
\[
F\colon X\times Y\rightarrow\mathbb{F}
\]
is defined to be the smallest $r$ such that 
\[
F(x,y)=\sum_{i=1}^{r}f_{i}(x)g_{i}(y)
\]
for some functions $f_{i},g_{i}$. The function $F$ is given by an
$|X|\times|Y|$ matrix with entries in $\mathbb{F}$, and the products
$f_{i}(x)g_{i}(y)$ correspond to the outer products of vectors. For
finite sets $X_{1},\dots,X_{n}$, a function 
\[
h\colon X_{1}\times\cdots\times X_{n}\rightarrow\mathbb{F}
\]
of the form 
\[
h(x_{1},\dots,x_{n})=f_{1}(x_{1})f_{2}(x_{2})\cdots f_{n}(x_{n})
\]
is called a \emph{rank 1 function,} and the tensor rank of is defined
to be the minimal $r$ such that 
\[
F=\sum_{i=1}^{r}g_{i}
\]
where the $g_{i}$ are rank 1 functions. Following the breakthrough
of Croot, Lev, and Pach \cite{CrootLevPachZ4}, Tao in his blog \cite{TaosBlogCapsets}
introduced the notion of the slice rank of a tensor:
\begin{defn}
Let $X_{1},\dots,X_{n}$ be finite sets. We say that the function
\[
h\colon X_{1}\times\cdots\times X_{n}\rightarrow\mathbb{F}
\]
has \emph{slice rank $1$ if}
\[
h(x_{1},\dots,x_{n})=f(x_{i})g\left(x_{1},\dots,x_{i-1},x_{i+1},\dots,x_{n}\right)
\]
for some $1\leq i\leq n$. The \emph{slice rank} of 
\[
F\colon X_{1}\times\cdots\times X_{n}\rightarrow\mathbb{F}
\]
is the smallest $r$ such that 
\[
F=\sum_{i=1}^{r}g_{i}
\]
where the $g_{i}$ have slice rank $1$.
\end{defn}
The following lemma was proven by Tao, and used to great effect:
\begin{lem}
Let $X$ be a finite set, and let $X^{n}$ denote the $n$-fold Cartesian
product of $X$ with itself. Suppose that 
\[
F\colon X^{n}\rightarrow\mathbb{F}
\]
is a diagonal tensor, that is 
\[
F(x_{1},\dots,x_{n})=\sum_{a\in A}c_{a}\delta_{a}(x_{1})\cdots\delta_{a}(x_{n})
\]
for some $A\subset X$, $c_{a}\neq0$, where 
\[
\delta_{a}(x)=\begin{cases}
1 & x=a\\
0 & \text{otherwise}
\end{cases}.
\]
Then 
\[
\slicerank(F)=|A|.
\]
\end{lem}
\begin{proof}
We refer the reader to \cite[Lemma 1]{TaosBlogCapsets} or \cite[Lemma 4.7]{BlasiakChurchCohnGrochowNaslundSawinUmans2016MatrixMultiplication}.
\end{proof}

\subsection{\label{subsec:The-Partition-Rank}The Partition Rank}

We introduce a new more general definition of the slice rank, that
we call the partition rank. Given variables $x_{1},\dots,x_{n}$ and
a set $S\subset\{1,\dots,n\}$, $S=\left\{ s_{1},\dots,s_{k}\right\} $,
we use the notation $\vec{x}_{S}$ to denote the subset of variables
\[
x_{s_{1}},\dots,x_{s_{k}},
\]
and so for a function $g$ of $k$ variables, we have that 
\[
g(\vec{x}_{S})=g(x_{s_{1}},\dots,x_{s_{k}}).
\]
For example, if $S=\{1,3,4\}$ and $T=\{2,5\}$ then 
\[
g\left(\vec{x}_{S}\right)f\left(\vec{x}_{T}\right)=g(x_{1},x_{3},x_{4})f(x_{2},x_{5}).
\]
A partition of $\{1,2,\dots,n\}$ is a collection $P$ of non-empty,
pairwise disjoint, subsets of $\{1,\dots,n\}$, satisfying 
\[
\bigcup_{A\in P}A=\{1,\dots,n\}.
\]
We say that $P$ is the \emph{trivial} partition if it consists only
of a single set, $\{1,\dots,n\}$. 
\begin{defn}
Let $X_{1},\dots,X_{n}$ be finite sets, and suppose that
\[
h:X_{1}\times\cdots\times X_{n}\rightarrow\mathbb{F}.
\]
If there exists some non-trivial partition $P$ such that 
\[
h(x_{1},\dots,x_{n})=\prod_{A\in P}f_{A}\left(\vec{x}_{A}\right)
\]
for some functions $f_{A}$, then $h$ is said to have \emph{partition
rank $1$.}
\end{defn}
Equivalently, the tensor $h:X_{1}\times\cdots\times X_{n}\rightarrow\mathbb{F}$
has partition rank $1$ if the variables can be split into disjoint
non-empty sets $S_{1},\dots,S_{t}$, with $t\geq2$, such that 
\[
S_{1}\cup\cdots\cup S_{t}=\{1,2,\dots,n\}
\]
and 
\[
h(x_{1},\dots,x_{n})=f_{1}(\vec{x}_{S_{1}})f_{2}(\vec{x}_{S_{2}})\cdots f_{n}(\vec{x}_{S_{t}})
\]
for some functions $f_{1},\dots,f_{t}$. In particular, $h:X_{1}\times\cdots\times X_{n}\rightarrow\mathbb{F}$
will have partition rank $1$ if and only if it can be written as
\[
h(x_{1},\dots,x_{n})=f(\vec{x}_{S})g(\vec{x}_{T})
\]
for some $f,g$ and some disjoint $S,T\neq\emptyset$ with $S\cup T=\{1,\dots,n\}$,
and $h$ will have slice rank $1$ if it can be written in the above
form with either $|S|=1$ or $|T|=1$. In other words, $h$ has partition
rank $1$ if the tensor can be written as a non-trivial outer product,
and it has slice rank $1$ if it can be written as the outer product
between a vector and a $k-1$ dimensional tensor.
\begin{example}
Let $X$ be a finite set. The function $h:X^{7}\rightarrow\mathbb{F}$
given by 
\[
h(x_{1},x_{2},x_{3},x_{4},x_{5},x_{6},x_{7})=f_{1}(x_{2},x_{5},x_{7})f_{2}(x_{1},x_{3})f_{3}(x_{4},x_{6})
\]
will have partition rank $1$, with partition $P$ given by the sets
with $S_{1}=\left\{ 2,5,7\right\} $, $S_{2}=\{1,3\}$ and $S_{3}=\{4,6\}$. 
\end{example}
This leads us to the definition of the partition rank:
\begin{defn}
Let $X_{1},\dots,X_{n}$ be finite sets.\emph{ The} \emph{partition
rank} of $F$ is the minimal $r$ such that 
\[
F=\sum_{i=1}^{r}g_{i}
\]
where the $g_{i}$ have partition rank $1$.
\end{defn}
The partition rank is the minimal rank among all possible ranks obtained
from partitioning or separating the variables. It will be convenient
to  have a notation for the rank corresponding to a specific subset
of partitions $\mathcal{P}$.
\begin{defn}
Let $X_{1},\dots,X_{n}$ be finite sets, and let $\mathcal{P}$ be
a collection of non-trivial partitions of $\{1,\dots,n\}$, and let
\[
h\colon X_{1}\times\cdots\times X_{n}\rightarrow\mathbb{F}.
\]
We say that $h$ has \emph{$\mathcal{P}$-rank $1$} if there exists
a partition $P\in\mathcal{P}$ such that 
\[
h(x_{1},\dots,x_{n})=\prod_{A\in P}f_{A}\left(\vec{x}_{A}\right)
\]
for some functions $f_{A}$. The $\mathcal{P}$\emph{-rank} of a function
\[
F\colon X_{1}\times\cdots\times X_{n}\rightarrow\mathbb{F},
\]
is defined to be the minimal $r$ such that 
\[
F=\sum_{i=1}^{r}g_{i}
\]
where the $g_{i}$ have $\mathcal{P}$-rank $1$.
\end{defn}
The partition rank is given by the $\mathcal{P}$-rank when $\mathcal{P}$
is the set of all non-trivial partitions. Let $\mathcal{P}_{\text{slice}}$
denotes the set of partitions of $\{1,\dots,n\}$ into a set of size
$1$ and a set of size $n-1$. Then the $\mathcal{P}_{\text{slice}}$-rank
will be equal to the slice-rank, and it follows that 
\[
\prank\leq\slicerank.
\]
Letting $\mathcal{P}_{\text{tensor}}$ denote the set containing only
the partition of $\{1,\dots,n\}$ into $n$ sets each of size $1$.
Then the $\mathcal{P}_{\text{tensor}}$-rank will be equal to the
tensor rank. This partition of $\{1,\dots,n\}$ is a refinement of
every partition in $\mathcal{P}_{\text{slice}}$, and so 
\[
\prank\leq\slicerank\leq\trank.
\]
Generalizing this relation between a refinement of a partition and
the $\mathcal{P}$-rank, we have the following proposition:
\begin{prop}
Let $\mathcal{P},\mathcal{P}'$ be two collections of non-trivial
partitions of $\{1,\dots,n\}$. Suppose that every partition $P\in\mathcal{P}$
is refined by some partition $P'\in\mathcal{P}'$. Then we have 
\[
\mathcal{P}\dashrank\leq\mathcal{P}'\dashrank.
\]
\end{prop}
In particular, this proposition implies that the partition rank is
equal to the bipartition rank, the $\mathcal{P}$-rank when $\mathcal{P}$
is the set of all bipartitions. 
\begin{proof}
Suppose that $g$ has $\mathcal{P}'$-rank $1$. Then there exists
$P'\in\mathcal{P}'$, and $f_{A}$ such that 
\[
h=\prod_{A\in P'}f_{A}.
\]
Since $P'$ refines some $P\in\mathcal{P}$, we may write
\[
h=\prod_{B\in P}g_{B},
\]
where 
\[
g_{B}=\prod_{\begin{array}{c}
A\in P\\
A\subset B
\end{array}}f_{A}.\qedhere
\]
\end{proof}
When $k=2$, the slice rank, partition rank and tensor rank are identical,
since there is only one non-trivial partition of $2$. When $k=3$,
the slice rank and partition rank are identical, but differ from the
tensor-rank, and when $k\geq4$, all three are different. However,
the partition rank can be substantially lower than the slice rank.
\begin{example}
Consider $k=4$. The only partitions of $\{1,2,3,4\}$ that do not
refine partitions appearing in $\mathcal{P}_{\text{slice }}$ are
given by the additive partition $2+2=4$, that is $\{\{1,2\},\{3,4\}\}$,
$\{\{1,3\},\{2,4\}\}$, $\{\{1,4\},\{2,3\}\}$. For a finite set $X$
and a field $\mathbb{F}$, consider the function 
\[
F\colon X\times X\times X\times X\rightarrow\mathbb{F}
\]
given by
\[
F(x,y,z,w)=\begin{cases}
1 & x=y\text{ and }z=w\\
0 & \text{otherwise}
\end{cases},
\]
that is 
\[
F(x,y,z,w)=\delta(x,y)\delta(z,w)
\]
where $\delta(x,y)$ is the function that is $1$ when $x=y$, and
$0$ otherwise. Then by definition $F$ satisfies $\prank(F)=1$.
Theorem 4.6 of \cite{BlasiakChurchCohnGrochowNaslundSawinUmans2016MatrixMultiplication}
states that if $F$ is a \emph{stable tensor} in the sense of Geometric
Invariant Theory, then $\slicerank(F)=|X|$. To see why $F$ is stable,
let $G=SL_{|X|}\times SL_{|X|}\times SL_{|X|}\times SL_{|X|}$ and
recall that a tensor $F$ is \emph{unstable }if and only if every
$G$-invariant homogeneous polynomial vanishes on $F$. Any function
$h:X\times X\times X\times X\rightarrow\mathbb{F}$ can be viewed
as a function 
\[
h\colon(X\times X)\times(X\times X)\rightarrow\mathbb{F},
\]
and so the $4$-tensor $h$ is the tensor product of two $|X|\times|X|$
matrices over $\mathbb{F}$. Consider the polynomial given by the
determinant. This will be a $G$-invariant polynomial, and furthermore
$\det(F)=1$, which implies that $F$ is stable, and hence $\slicerank(F)=|X|$. 
\end{example}
Our key observation is that the partition rank of a diagonal tensor
is maximal.
\begin{lem}
\label{lem:Critical_Lemma}Let $X$ be a finite set, and let $X^{n}$
denote the $n$-fold Cartesian product of $X$ with itself. Suppose
that 
\[
F\colon X^{n}\rightarrow\mathbb{F}
\]
is a diagonal tensor, that is 
\[
F(x_{1},\dots,x_{n})=\sum_{a\in A}c_{a}\delta_{a}(x_{1})\cdots\delta_{a}(x_{n})
\]
for some $A\subset X$ where $c_{a}\neq0$, and 
\[
\delta_{a}(x)=\begin{cases}
1 & x=a,\\
0 & \text{otherwise}.
\end{cases}
\]
Then 
\[
\prank(F)=|A|.
\]
\end{lem}
\begin{proof}
It's evident that the partition rank is at most $|A|$, and so our
goal is to prove the lower bound. The proof proceeds by induction
on the number of variables. When $n=2$, this is the usual notion
of rank, and so the result follows. Suppose that $F$ has partition
rank $r<|A|$, that is suppose that we can write 
\[
F(x_{1},\dots,x_{n})=\sum_{i=1}^{r}f_{i}(\vec{x}_{S_{i}})g_{i}(\vec{x}_{T_{i}})
\]
for some sets $S_{i},T_{i}$ with $S_{i}\cap T_{i}=\emptyset$ and
$S_{i}\cup T_{i}=\{1,\dots,n\}$. Assume without loss of generality
that $|S_{i}|\leq\frac{n}{2}$ for each $i$. If there is no $i$
such that $|S_{i}|=1$, then choose an arbitrary variable, say $x_{1}$,
and average over that coordinate. Then 
\[
\sum_{x_{1}\in X}F(x_{1},\dots,x_{n})=\sum_{a\in A}c_{a}\delta_{a}(x_{2})\cdots\delta_{a}(x_{n})=\sum_{i=1}^{r}\tilde{f}_{i}(\vec{x}_{S_{i}\backslash\{1\}})\tilde{g}_{i}\left(\vec{x}_{T_{i}\backslash\{1\}}\right),
\]
for functions $\tilde{f},\tilde{g}$ given by averaging $f,g$ over
$x_{1}$. This contradicts the inductive hypothesis since $\sum_{a\in A}c_{a}\delta_{a}(x_{2})\cdots\delta_{a}(x_{n})$
will have partition rank equal to $|A|>r$. 

Suppose that there exists some $S_{i}$ such that $|S_{i}|=1$. Then
$S_{i}=\{j\}$ for some $j\in\{1,\dots,n\}$. Let $U$ be the set
of indices $u$ for which $S_{u}=\{j\}$. Consider the annihilator
of $U$, defined to be 
\[
V=\left\{ h\colon X\rightarrow\mathbb{F}:\ \sum_{x_{j}\in X}f_{u}(x_{j})h(x_{j})=0\text{ for all }u\in U\right\} .
\]
This vector space has dimension at least $|X|-|U|$, and this will
be positive since $|U|\leq r<|A|\leq|X|$. Let $v\in V$ have maximal
support, and set $\Sigma=\left\{ x\in X:\ v(x)\neq0\right\} $. Then
$|\Sigma|\geq\dim V\geq|X|-|U|$, since otherwise there exists nonzero
$w\in V$ vanishing on $\Sigma$, and the function $v+w$ would have
a larger support than $v$. Multiplying both sides of our expression
by $v(x_{j})$ and summing over $x_{j}$ reduces the dimension by
$1$. Indeed
\[
\sum_{x_{j}\in X}v(x_{j})F(x_{1},\dots,x_{n})=\sum_{a\in A}c_{a}\delta_{a}(x_{1})\cdots\delta_{a}(x_{j-1})\delta_{a}(x_{j+1})\cdots\delta_{a}(x_{n})\left(\sum_{x_{j}\in X}v(x_{j})\delta_{a}(x_{j})\right),
\]
and since the sum $\sum_{x_{j}\in X}v(x_{j})\delta_{a}(x_{j})$ will
be non-zero for at least $|X|-|U|$ values of $a\in X$, the partition
rank of the above must be at least $|A|-|U|$ by the inductive hypothesis.
Since 
\[
\sum_{x_{j}\in X}v(x_{j})f_{i}(\vec{x}_{S_{i}})=0
\]
for each $i\in U$, it follows that 
\[
\sum_{x_{j}\in X}v(x_{j})\sum_{i=1}^{r}c_{i}f_{i}(\vec{x}_{S_{i}})g_{i}(\vec{x}_{T_{i}})
\]
will be a sum of at most $k-|U|$ partition rank $1$ functions, and
hence it has partition rank at most $k-|U|$. This implies that $|A|-|U|<k-|U|$,
which is a contradiction, and the lemma is proven.
\end{proof}
Since the partition rank is minimal over all $\mathcal{P}$-ranks,
the $\mathcal{P}$-rank for any set of non-trivial partitions $\mathcal{P}$
of a diagonal tensor will equal to the number of non-zero entries
on that diagonal tensor.

We conclude this section by rephrasing an open problem in non-commutative
circuits in terms of the $\mathcal{P}$-rank:
\begin{problem}
\label{prob:circuit_open}Let $X$ be a finite set of size $n$, and
let $k=4$. Let $\mathcal{P}$ be the set that contains the two partitions
$\{\{1,2\},\{3,4\}\}$, $\{\{1,3\},\{2,4\}\}$. Does $\delta(x_{1},x_{4})\delta(x_{2},x_{3})$
have superlinear $\mathcal{P}$-rank? That is, does $\delta(x_{1},x_{4})\delta(x_{2},x_{3})$
have $\mathcal{P}\dashrank\gg n^{1+\epsilon}$ for some $\epsilon>0$
as $n$ grows?
\end{problem}
A counting argument shows that there will exist many tensors of $\mathcal{P}$-rank
$\gg n^{2}$, however no explicit superlinear lower bounds are known
for any tensor. A positive answer to problem \ref{prob:circuit_open}
would lead to improved lower bounds for non-commutative circuits,
see \cite[Theorem 3.6]{ShpilkaYehudayoffCircuitsSurvey}. In general,
we can ask about the $\mathcal{P}$-rank of a product of $\delta$
functions that is given by a partition that is not a refinement of
any $P\in\mathcal{P}$.
\begin{problem}
\label{prob:refinement_lower} Let $X$ be a finite set. Let $\mathcal{P}$
be a collection of non-trivial partitions of $\{1,\dots,n\}$ and
suppose that $P$ is not a refinement of any $P'\in\mathcal{P}$.
Let 
\begin{equation}
\delta_{P}(x_{1},\dots,x_{n})=\prod_{A\in P}\delta\left(\vec{x}_{A}\right),\label{eq:delta_P_def}
\end{equation}
where for a singleton set $A=\{j\}$, we use the convention $\delta\left(\vec{x}_{A}\right)=\delta\left(x_{j}\right)=1.$
What is the $\mathcal{P}$-rank of $\delta_{P}$? 
\end{problem}

\section{The Distinctness Indicator Function\label{sec:The-Distinctness-Indicator}}

Let $X$ be a finite set, $\mathbb{F}$ a field, and let $X^{k}=X\times\cdots\times X$
denote the Cartesian product of $X$ with itself $k$ times. For every
$\sigma\in S_{k}$, define 
\[
f_{\sigma}\colon X\times\cdots\times X\rightarrow\mathbb{F}
\]
to be the function that is $1$ if $(x_{1},\dots,x_{k})$ is a fixed
point of $\sigma$, and $0$ otherwise. Using these functions $f_{\sigma}$
we will construct an indicator function for the distinctness of $k$
variables.
\begin{lem}
\label{lem:Group-action-lemma}We have the identity
\[
\sum_{\sigma\in S_{k}}\sgn(\sigma)f_{\sigma}(x_{1},\dots,x_{k})=\begin{cases}
1 & \text{if }x_{1},\dots,x_{k}\text{ are distinct},\\
0 & \text{otherwise},
\end{cases}
\]
where $\sgn(\sigma)$ is the sign of the permutation\@.
\end{lem}
\begin{proof}
By definition, 
\[
\sum_{\sigma\in S_{k}}\sgn(\sigma)f_{\sigma}(x_{1},\dots,x_{k})=\sum_{\sigma\in\text{Stab}(\vec{x})}\sgn(\sigma)
\]
where $\text{Stab}(\vec{x})\subset S_{k}$ is the stabilizer of $\vec{x}$.
Since the stabilizer is a product of symmetric groups, this will be
non-zero precisely when $\text{Stab}(\vec{x})$ is trivial, and hence
$x_{1},\dots,x_{k}$ must be distinct. This vector is then fixed only
by the identity element, and so the sum equals $1$.
\end{proof}
\begin{lem}
\label{lem:critical-indicator-function}Let $\text{Cyc}\subset S_{k}$,
be the $k$-cycles in $S_{k}$, and define 
\[
H_{k}(x_{1},\dots,x_{k})=\sum_{\begin{array}{c}
\sigma\in S_{k}\\
\sigma\notin\text{Cyc}
\end{array}}\sgn(\sigma)f_{\sigma}(x_{1},\dots,x_{k}).
\]
Then 
\[
H_{k}(x_{1},\dots,x_{k})=\begin{cases}
1 & \text{if }x_{1},\dots,x_{k}\text{ are distinct},\\
(-1)^{k-1}(k-1)! & \text{if }x_{1}=\cdots=x_{k},\\
0 & \text{otherwise}.
\end{cases}
\]
\end{lem}
\begin{proof}
This follows from lemma \ref{lem:Group-action-lemma} and the fact
that the conjugacy class of the $k$-cycle in $S_{k}$ is precisely
the set of all $k$-cycles, and this has $(k-1)!$ elements with sign
equal to $(-1)^{k-1}$.
\end{proof}
This function can be used to zero-out those tuples of vectors with
repetitions. Suppose that we are interested in the size of the largest
set $A\subset X$ that does not contain $k$ distinct vectors satisfying
some condition $\mathcal{K}$. Then if $F_{k}:X^{k}\rightarrow\mathbb{F}$
is some function satisfying
\[
F_{k}(x_{1},\dots,x_{k})=\begin{cases}
c_{1} & \text{if }x_{1},\dots,x_{k}\text{ satisfy }\mathcal{\mathcal{K}}\\
c_{2} & \text{if }x_{1}=\cdots=x_{k}\\
0 & \text{otherwise}
\end{cases}
\]
where $c_{2}\neq0$, then 
\[
I_{k}(x_{1},\dots,x_{k})\coloneqq F_{k}(x_{1},\dots,x_{k})H_{k}(x_{1},\dots,x_{k})
\]
when restricted to $A^{k}$ will be a diagonal tensor, and hence by
lemma \ref{lem:Critical_Lemma}
\[
|A|\leq\prank(I_{k}).
\]
Let 
\begin{equation}
\delta(x_{1},\dots,x_{k})=\begin{cases}
1 & \text{if }x_{1}=\cdots=x_{k}\\
0 & \text{otherwise}
\end{cases}\label{eq:delta_def}
\end{equation}
and for a partition $P$ of $\{1,\dots,k\}$ define $\delta_{P}(x_{1},\dots,x_{k})$
as in (\ref{eq:delta_P_def}). Suppose that $\sigma$ is a permutation
of $\{1,\dots,k\}$ with $t$ cycles in its disjoint cycle decomposition,
and these cycles permute the sets $S_{1},S_{2},\dots,S_{t}$. Then
$P=\{S_{1},\dots,S_{t}\}$ will be a partition of $\{1,\dots,k\}$
and we have that 
\begin{equation}
f_{\sigma}=\delta_{P}.\label{eq:f_sigma_delta_P}
\end{equation}
It follows that $H_{k}(x_{1},\dots,x_{k})$ can be written as a sum
product of delta functions. When $k=2$, we have that 
\[
H_{2}(x_{1},x_{2})=1,
\]
when $k=3$
\[
H_{2}(x_{1},x_{2},x_{3})=1-\delta(x_{1},x_{2})-\delta_{2}(x_{2},x_{3})-\delta(x_{3},x_{2})
\]
and when $k=4$
\begin{align*}
H_{3}(x_{1},x_{2},x_{3},x_{4})= & 1-\delta(x_{1},x_{2})-\delta_{2}(x_{2},x_{3})-\delta(x_{3},x_{4})-\delta(x_{4},x_{1})-\delta(x_{1},x_{3})-\delta(x_{2},x_{4})\\
 & +2\delta(x_{1},x_{2},x_{3})+2\delta(x_{2},x_{3},x_{4})+2\delta(x_{3},x_{4},x_{1})+2\delta(x_{4},x_{1},x_{2})\\
 & +\delta(x_{1},x_{2})\delta(x_{3},x_{4})+\delta(x_{1},x_{3})\delta(x_{2},x_{4})+\delta(x_{1},x_{4})\delta(x_{2},x_{3}).
\end{align*}
Note that as a function on $X^{3}$, $\slicerank(\delta(x_{1},x_{2}))=1$,
but on $X^{4}$ the function
\[
\delta(x_{1},x_{2})\delta(x_{3},x_{4})
\]
has slice rank equal to $|X|$. This function \emph{does} however
have partition rank equal to $1$. When writing $H_{k}(x_{1},\dots,x_{k})$
as a linear combination of $\delta$ functions, it will always have
partition rank at most $2^{k}-1$ since every term will have a $\delta$-function
corresponding to some subset strict $S\subset\{1,\dots,k\}$. Starting
from $k=4$, there will be functions in the sum whose slice rank is
maximal, and this is why the partition rank is needed to handle a
function of the form 
\[
I_{k}(x_{1},\dots,x_{k})\coloneqq F_{k}(x_{1},\dots,x_{k})H_{k}(x_{1},\dots,x_{k}).
\]

\section{Right Angles in $\mathbb{F}_{q}^{n}$\label{sec:Right-Angles}}

In this section we provide a proof of theorem \ref{thm:right-angle}.
We begin with a lemma concerning the dimension of the space of polynomials
of degree $\leq d$ in $n$ variables over $\mathbb{F}_{q}$. 

\subsection{The Space of Polynomials}

Let $\Poly_{D}\left(\mathbb{F}^{n}\right)$ denote the space of $n$-variable
polynomials over a field $\mathbb{F}$ of degree at most $D$. The
dimension of this space, $\dim\Poly_{D}\left(\mathbb{F}^{n}\right)$,
is equal to the number of $n$-variable monomials of degree $\leq D$
over $\mathbb{F}$. 
\begin{lem}
\label{lem:balls_urns}We have that 
\begin{equation}
\dim\Poly_{D}\left(\mathbb{F}^{n}\right)=\binom{D+n}{n}.\label{eq:balls-and-urns-counting-bound}
\end{equation}
\end{lem}
\begin{proof}
The balls and urns theorem states that there are 
\[
\binom{j+n-1}{n}
\]
 ways to put $j$ balls into $n$ urns. Viewing the balls as the exponents,
and the urns as the monomials $x_{1},\dots,x_{n}$, by summing over
$j$, we arrive at
\[
\sum_{j=0}^{D}\binom{j+n-1}{n}=\binom{D+n}{n}.
\]
\end{proof}
The specific case of a polynomial ring generated by $x_{1},\dots,x_{n}$
and $(x_{1}^{2}+\cdots+x_{n}^{2})$, is of particular interest, and
we define 
\begin{equation}
\Poly_{D}^{2}(\mathbb{F}^{n})=\spann_{\mathbb{F}}\left\{ (x_{1}^{2}+\cdots+x_{n}^{2})^{a_{0}}x_{1}^{a_{1}}\cdots x_{n}^{a_{n}}|\ \sum_{i=0}^{n}a_{0}\leq D\right\} .\label{eq:Poly_2_def}
\end{equation}
Equation (\ref{eq:balls-and-urns-counting-bound}) gives the upper
bound
\begin{equation}
\dim\Poly_{d}^{2}(\mathbb{F}^{n})\leq\binom{n+1+d}{d},\label{eq:Poly_2_counting_bound}
\end{equation}
but we will make use of the exact count given by Bannai and Bannai
\cite{BannaiBannai1981sDistanceSubset}:
\begin{lem}
\label{lem:BBB}(Bannai and Bannai \cite[equation 3]{BannaiBannai1981sDistanceSubset})
Let $\Poly_{d}^{2}(\mathbb{F}^{n})$ be as defined in (\ref{eq:Poly_2_def}).
Then
\[
\dim\Poly_{d}^{2}(\mathbb{F}^{n})=\binom{n+d}{d}+\binom{n+d-1}{d-1}.
\]
\end{lem}
\begin{rem}
When $\mathbb{F}=\mathbb{F}_{q}$ is a finite field, the space of
functions corresponding to the space of polynomials of degree at most
$D$ will be smaller since $x^{q}=x$. In our applications, we are
interested only in bounding the dimension of the space of functions,
rather than as a space of formal polynomials. By using a Chernoff
type bound, one can obtain improved counting bounds when $D$ is large,
say $D=\Omega(n)$. However, in the applications that follow, $D$
may not be very large relative to $n$, and so Lemma \ref{lem:balls_urns}
and Lemma \ref{lem:BBB} will be sufficient. 
\end{rem}
Using this lemma, and the slice rank, we prove theorem \ref{thm:right-angle}:
\begin{proof}[\emph{Proof of Theorem \ref{thm:right-angle}}]
\emph{ }Consider the function

\[
F:\mathbb{F}_{q}^{n}\times\mathbb{F}_{q}^{n}\times\mathbb{F}_{q}^{n}\rightarrow\mathbb{F}_{q}
\]
 defined by 
\[
F(x,y,z)=\left(1-\delta(x,y)-\delta(y,z)-\delta(z,x)\right)\left(1-\langle x-z,y-z\rangle^{q-1}\right).
\]
Then $F$ is an indicator for distinct right corners since 
\[
F(x,y,z)=\begin{cases}
-2 & \text{if }x=y=z,\\
1 & \text{if }x,y,z\text{ are distinct and form a right corner},\\
0 & \text{otherwise}.
\end{cases}
\]
If $A\subset\mathbb{F}_{q}^{n}$ has no right-corners, then $F|_{A\times A\times A}$
is a diagonal tensor, and so by Lemma \ref{lem:Critical_Lemma} 
\[
|A|\leq\slicerank(F).
\]
Note that
\[
\delta(x,z)\left(1-\langle x-z,y-z\rangle^{q-1}\right)=\delta(x,z),
\]
and
\[
\delta(y,z)\left(1-\langle x-z,y-z\rangle^{q-1}\right)=\delta(y,z),
\]
and that both of these functions have slice rank $1$. By definition
of the delta function, we have that 
\[
\delta(x,y)\left(1-\langle x-z,y-z\rangle^{q-1}\right)=\delta(x,y)\left(1-\langle x-z,x-z\rangle^{q-1}\right),
\]
and by expanding $\langle x-z,x-z\rangle^{q-1}$ as a degree $2(q-1)$
polynomial, the above can be written as a linear combination of terms
of the form 
\[
\delta(x,y)x_{0}^{d_{0}}x_{1}^{d_{1}}\cdots x_{n}^{d_{n}}z_{0}^{e_{0}}z_{1}^{e_{1}}\cdots z_{n}^{e_{n}}
\]
where $x_{0}=(x_{1}^{2}+\cdots+x_{n}^{2})$, $z_{0}=(z_{1}^{2}+\cdots+z_{n}^{2})$,
and 
\[
\sum_{i=0}^{n}e_{i}\leq q-1.
\]
Similarly, we may expand the degree $2(q-1)$ polynomial
\[
\left(1-\langle x-z,y-z\rangle^{q-1}\right)
\]
as a linear combination of terms of the form 
\[
x_{1}^{d_{1}}\cdots x_{n}^{d_{n}}z_{0}^{e_{0}}z_{1}^{e_{1}}\cdots z_{n}^{e_{n}}y_{1}^{f_{1}}\cdots y_{n}^{f_{n}},
\]
where once again $\sum_{i=1}^{n}e_{i}\leq q-1.$ In other words, both
of these terms can be expanded as a linear combination of terms of
the form 
\[
f(x,y)h(z)
\]
where $h\in\Poly_{q-1}^{2}(\mathbb{F}_{q}^{n})$. Hence by always
slicing off the $z$ coordinate, and decomposing $\delta(x,y)\langle x-z,x-z\rangle^{q-1}$
and $\langle x-z,y-z\rangle^{q-1}$ simultaneously, we obtain the
bound 
\[
\slicerank(F)\leq\dim\Poly_{q-1}^{2}(\mathbb{F}_{q}^{n})+2.
\]
By lemma \ref{lem:BBB} it follows that
\[
\slicerank(F)\leq\binom{n+q-1}{q-1}+\binom{n+q-2}{q-2}+2.
\]
Due to the binomial identity
\[
\binom{n+q-2}{q-3}+\binom{n+q-2}{q-2}+\binom{n+q-1}{q-1}=\binom{n+q}{q-1},
\]
the bound can be rearranged as 
\[
\slicerank(F)\leq\binom{n+q}{q-1}+2-\binom{n+q-2}{q-3}
\]
and the result follows.
\end{proof}

\section{$k$-Right Corners in $\mathbb{F}_{q}^{n}$ \label{sec:Right-Corners}}

In this section, we use the partition rank to prove theorem \ref{thm:Right-Corner}.
Our goal is to construct a function 
\[
J_{k}\colon\mathbb{F}_{q}^{k+1}\rightarrow\mathbb{F}_{q},
\]
of low partition rank satisfying 
\begin{equation}
J_{k}\left(x_{1},\dots,x_{k+1}\right)=\begin{cases}
c_{1} & x_{1},\dots,x_{k+1}\text{ form a \ensuremath{k} right corner},\\
c_{2} & x_{1}=\cdots=x_{k},\\
0 & \text{otherwise},
\end{cases}\label{eq:desired_properties}
\end{equation}
where $c_{2}\neq0$. Suppose that $E\subset\mathbb{F}_{q}$ is a set
without any $k$-right corners. Then $J_{k}$ restricted to $E^{k+1}$
will be a diagonal tensor taking the value $c_{2}$ on the diagonal,
and so by lemma \ref{lem:Critical_Lemma}, $J_{k}$ must have partition
rank at least $|E|$. It follows that any set $E\subset\mathbb{F}_{q}$
of size 
\[
E>\prank(J_{k}),
\]
must contain a $k$-right corner. Define 
\[
F_{k}\colon\mathbb{F}_{q}^{k}\rightarrow\mathbb{F}_{q}
\]
by 
\begin{equation}
F_{k}(x_{1},\dots,x_{k})=\prod_{j<l\leq k}\left(1-\langle x_{j},x_{l}\rangle^{q-1}\right).\label{eq:F_def}
\end{equation}
This polynomial has degree 
\begin{equation}
\deg F=k(k-1)(q-1)\label{eq:degF}
\end{equation}
and satisfies 
\[
F_{k}\left(x_{1},\dots,x_{k}\right)=\begin{cases}
1 & x_{1},\dots,x_{k}\text{ are mutually orthogonal},\\
0 & \text{otherwise}.
\end{cases}
\]
Note that $x_{1},\dots,x_{k}$ are not required to be unique for $F_{k}(x_{1},\dots,x_{k})$
to equal $1$. Define 
\begin{equation}
R_{k+1}(x_{1},\dots,x_{k+1})=F_{k}(x_{1}-x_{k+1},x_{2}-x_{k+1},\dots,x_{k}-x_{k+1})\label{eq:R_k_definition}
\end{equation}
and observe that 
\[
R_{k+1}(x_{1},\dots,x_{k})=\begin{cases}
1 & x_{1}-x_{k+1},\dots,x_{k}-x_{k+1}\text{ are mutually orthogonal},\\
0 & \text{otherwise}.
\end{cases}
\]
This function is not an indicator function for $k$-right corners
since there are trivial off-diagonal solutions that arise from repeated
variables. For example, suppose that $x_{1}-x_{k+1}$ and $x_{2}-x_{k+1}$
are orthogonal, and that $x_{2}-x_{k+1}$ is self orthogonal. Then
if we take $x_{i}=x_{2}$ for $2\leq i\leq k$, the $(k+1)$-tuple
$(x_{1},\dots,x_{k+1})$ will satisfy $R(x_{1},\dots,x_{k+1})=1$
despite not being a $k$-right corner. To handle this issue, we use
the distinctness indicator function $H_{k+1}$ from lemma \ref{lem:critical-indicator-function},
and define
\begin{equation}
J_{k}=H_{k+1}R_{k+1}\label{eq:J_k definition}
\end{equation}
The function $J_{k}$ will satisfy
\begin{equation}
J_{k}(x_{1},\dots,x_{k+1})=\begin{cases}
1 & \text{if }x_{1},\dots,x_{k+1}\text{ are distinct and form a }k\text{-right corner},\\
(-1)^{k}k! & \text{if }x_{1}=\cdots=x_{k+1},\\
0 & \text{otherwise}.
\end{cases}\label{eq:J_k_as_indicator_function}
\end{equation}
If $q=p^{r}$ with $p>k$, then $(-1)^{k}k!\neq0$, and $J_{k}$ satisfies
equation (\ref{eq:desired_properties}). The following lemma follows
from the work presented thus far:
\begin{lem}
\label{lem:thm_from_J_k_prank_bound}Let $q=p^{r}$ with $p>k$, and
let $J_{k}$ be defined according to (\ref{eq:J_k definition}). If
$A\subset\mathbb{F}_{q}^{n}$ contains no $k$-right corners, then
\begin{equation}
|A|\leq\prank(J_{k}).\label{eq:A_prank_J_k_bound}
\end{equation}
\end{lem}
To prove theorem \ref{thm:Right-Corner}, what remains is to bound
the partition rank of $J_{k}$. 

\subsection{Decomposition of $J_{k}$}

In this subsection we will decompose $J_{k}$ as a linear combination
of terms, and use this decomposition to prove the following partition
rank bound for $J_{k}$:
\begin{prop}
\label{prop:J_k_prank_bound}We have that 
\[
\prank(J_{k})\leq\binom{n+(k-1)q}{(k-1)(q-1)}.
\]
\end{prop}
Let $\mathcal{P}_{k+1}$ denote the set of non-trivial partitions
of $\{1,\dots,k+1\}$. By lemma \ref{lem:critical-indicator-function},
and (\ref{eq:f_sigma_delta_P}), $H_{k+1}$ can be written as a sum
product of delta functions 
\[
H_{k+1}=\sum_{P\in\mathcal{P}_{k+1}}c_{P}\prod_{A\in P}\delta\left(\vec{x}_{A}\right)
\]
where $c_{P}$ is a constant depending on the partition $P$. In the
construction of $H_{k+1}$, only non-trivial partitions were included,
and so we can split the product $J_{k}=H_{k+1}R_{k+1}$ into a linear
combination of terms of the form 
\begin{equation}
R_{k+1}(x_{1},\dots,x_{k+1})\prod_{A\in P}\delta\left(\vec{x}_{A}\right)\label{eq:RkA_term}
\end{equation}
where the product of delta functions term always contains two or more
delta functions. For such a term, the product of delta functions forces
many variables among $x_{1},\dots,x_{k+1}$ to be equal. The following
lemma uses the fact that equality among $x_{1},\dots,x_{k+1}$ introduces
redundant variables into the product definition of $R_{k+1}$ to express
$R_{k+1}$ as a lower degree polynomial.
\begin{lem}
\label{lem:R_k_simplification}Let $P=\{A_{1},\dots,A_{r}\}$ be a
non-trivial paritition of $\{1,\dots,k+1\}$. Let $a_{i}=\min(A_{i})$
denote the minimal element of $A_{i}$ for each $i\leq r-1$, and
suppose without loss of generality that $k+1\in A_{r}$, and let $a_{r}=k+1$.
Then 
\begin{equation}
R_{k+1}(x_{1},\dots,x_{k+1})\prod_{i=1}^{r}\delta\left(\vec{x}_{A_{i}}\right)=R_{r}(x_{a_{1}},\dots,x_{a_{r}})\prod_{i=1}^{r}\delta\left(\vec{x}_{A_{i}}\right)\Pi_{2}^{P}\label{eq:R_k_simplification}
\end{equation}
where 
\[
\Pi_{2}^{P}=\prod_{\substack{|A_{i}|>1\\
i\neq r
}
}\left(1-\langle x_{a_{i}}-x_{a_{r}},x_{a_{i}}-x_{a_{r}}\rangle^{q-1}\right)
\]
\end{lem}
\begin{proof}
For a set $A$ and any element $a\in A$, and a polynomial function
$Q$ of $|A|$ variables, the product $\delta(\vec{x}_{A})Q(\vec{x}_{A})$
will be exactly equal to $\delta(\vec{x}_{A})Q(x_{a},x_{a},\dots,x_{a})=\delta(\vec{x}_{A})\tilde{Q}(x_{a})$,
where $\tilde{Q}$ is a single variable polynomial obtained from $Q$
by making all the variables equal. For our purposes, $a_{i}$ need
only be some representative member of $A_{i}$, where the choice is
made in a way that is consistent across different partitions. For
simplicity, we choose $a_{i}=\min(A_{i})$ for $i\leq r-1$ and $a_{r}=k+1$. 

By the definition of $R_{k+1}$, (\ref{eq:RkA_term}) can be written
as
\[
\prod_{i=1}^{r}\delta\left(\vec{x}_{A_{i}}\right)\prod_{1\leq j<l\leq k}\left(1-\langle x_{j}-x_{k+1},x_{l}-x_{k+1}\rangle^{q-1}\right).
\]
Due to the delta functions, and the equality 
\[
\left(1-\langle x,y\rangle^{q-1}\right)^{2}=\left(1-\langle x,y\rangle^{q-1}\right),
\]
we can set many of the variables above to be equal without changing
the function. Consequently there will be many redundant terms in the
product, and it follows that 
\[
R_{k+1}(x_{1},\dots,x_{k+1})\prod_{i=1}^{r}\delta\left(\vec{x}_{A_{i}}\right)=R_{r}(x_{a_{1}},\dots,x_{a_{r}})\prod_{i=1}^{r}\delta\left(\vec{x}_{A_{i}}\right)\Pi_{2}^{P}.
\]
The additional product 
\[
\Pi_{2}^{P}=\prod_{\substack{i<r\ \text{s.t.}\\
|A_{i}|\geq2
}
}\left(1-\langle x_{a_{i}}-x_{a_{r}},x_{a_{i}}-x_{a_{r}}\rangle^{q-1}\right)
\]
arises due to the fact that forcing the equality $x_{i}=x_{j}$ does
not make $\langle x_{i}-x_{k+1},x_{j}-x_{k+1}\rangle$ equal to zero. 
\end{proof}
Let $\Poly_{d}(\mathbb{F}_{q}^{n})$ and $\Poly_{d}^{2}(\mathbb{F}_{q}^{n})$
be the polynomial spaces defined in section \ref{sec:Right-Angles}.
Applying equation (\ref{eq:R_k_simplification}), we have the following
lemma:
\begin{lem}
\label{lem:R_k_partition_degree_bound}Let $P=\{A_{1},\dots,A_{r}\}$
be a non-trivial paritition of $\{1,\dots,k+1\}$. Let $a_{i}=\min(A_{i})$
for each $i\leq r-1$, and suppose without loss of generality that
$k+1\in A_{r}$, and let $a_{r}=k+1$. Then we have that
\[
R_{k+1}(x_{1},\dots,x_{k+1})\prod_{i=1}^{r}\delta\left(\vec{x}_{A_{i}}\right)=\sum_{j}\left[\prod_{i=1}^{r}\delta\left(\vec{x}_{A_{i}}\right)Q_{i,j}\left(x_{a_{i}}\right)\right]
\]
where for each $1\leq i\leq r-1$, 
\[
Q_{i,j}\in\begin{cases}
\Poly_{d}(\mathbb{F}_{q}^{n})\ \text{with }d=(r-2)(q-1) & \text{ if }|A_{i}|=1,\\
\Poly_{d}^{2}(\mathbb{F}_{q}^{n})\ \text{with }d=(r-1)(q-1) & \text{ if }|A_{i}|>1.
\end{cases}
\]
\end{lem}
\begin{proof}
By definition,
\[
R_{r}(x_{a_{1}},\dots,x_{a_{r}})=\prod_{1\leq j<l\leq r-1}\left(1-\langle x_{a_{j}}-x_{a_{r}},x_{a_{l}}-x_{a_{r}}\rangle^{q-1}\right),
\]
and so by expanding the product in equation (\ref{eq:R_k_simplification})
of Lemma \ref{lem:R_k_simplification}, we may write 
\[
R_{r}(x_{a_{1}},\dots,x_{a_{r}})\prod_{i=1}^{r}\delta\left(\vec{x}_{A_{i}}\right)\Pi_{2}^{P}=\sum_{j}\left[\prod_{i=1}^{r}\delta\left(\vec{x}_{A_{i}}\right)Q_{i,j}\left(x_{a_{i}}\right)\right]
\]
for some polynomials $Q_{j,i}$. To prove the lemma, we need to prove
degree bounds on the $Q_{i,j}$. For any $1\leq i\leq r-1$, there
are exactly $r-2$ terms each of degree $q-1$ in the product definition
of $R_{r}$ that contain $x_{a_{i}}$. If $|A_{i}|=1$, then $x_{a_{i}}$
does not appear in $\Pi_{2}^{P}$, which implies that $\deg Q_{i,j}\leq(r-2)(q-1)$,
and hence 
\[
Q_{i,j}\in\Poly_{d}(\mathbb{F}_{q}^{n})\ \ \ \text{with }d=(r-2)(q-1).
\]
If $|A_{i}|>1$, then the product 
\[
\prod_{\substack{i<r\ \text{s.t.}\\
|A_{i}|\geq2
}
}\left(1-\langle x_{a_{i}}-x_{a_{r}},x_{a_{i}}-x_{a_{r}}\rangle^{q-1}\right)
\]
will contain an additional term that contains $x_{a_{i}}$. This term
is of degree $2(q-1)$ in $x_{a_{i},1},\dots,x_{a_{i},n}$, and so
\[
Q_{i,j}\in\Poly_{d}(\mathbb{F}_{q}^{n})\ \ \ \text{with }d=r(q-1).
\]
For $|A_{i}|>1$, if instead we expand the additional product in terms
of the variables $x_{a_{i},1},\dots,x_{a_{i},n}$ as well as $x_{a_{i},1}^{2}+\cdots+x_{a_{i},n}^{2}$,
then it will have degree $q-1$. When written in these $n+1$ variables,
it follows that $\deg Q_{i,j}\leq(r-1)(q-1)$, and hence
\[
Q_{i,j}\in\Poly_{d}^{2}(\mathbb{F}_{q}^{n})\ \ \ \text{with }d=(r-1)(q-1).\qedhere
\]
\end{proof}
Note that the previous Lemma did not specify a degree bound for $Q_{i,j}$
with $i=r$. This is because many terms contain $x_{a_{r}}$, and
the degree could be very large. Consequently, to bound the Partition
Rank of $J_{k}$, for any Partition $P$ we will never use the delta
function term associated with the set $A_{r}\in P$ that contains
$k+1$. , 

One difficulty in bounding the partition rank of $J_{k}=H_{k+1}R_{k+1}$
by decomposing $J_{k}$ into pieces is handling all of the different
delta functions that arise from all of the different partitions. A
naive approach would introduce a large factor (depending on $k$)
into our final bound. We need to make use of the fact that the delta
function associated with any particular set $B\subset\{1,\dots,k+1\}$
will appear in multiple different partitions of $\{1,\dots,k+1\}$.
In the following proposition we take a representative sample of subsets
that must appear in every partition, and use this to bound the partition
rank of $J_{k}$.
\begin{prop}
\label{prop:sequence_of_sets_bound}Let $\mathcal{P}_{k+1}$ denote
the set of non-trivial partitions of $\{1,\dots,k+1\},$and for each
$P\in\mathcal{P}_{k+1}$ let $r(P)$ denote the number of members
in $P$. Suppose that $B_{1},B_{2},\dots,B_{l}$ is a sequence of
subsets of $\{1,\dots,k\}$ such that for every $P\in\mathcal{P}_{k+1}$
there exists $A\in P$, where $k+1\notin A$, and $A=B_{i}$ for some
$i$. For each $i$, set 
\begin{equation}
r_{i}=\max\left\{ r(P):\ P\in\mathcal{P}_{k+1},\ B_{i}\in P,\text{ and }B_{j}\notin P\text{ for }j<i\right\} ,\label{eq:partition_max_criteria}
\end{equation}
and let 
\[
\mathcal{V}_{i}=\begin{cases}
\Poly_{(r_{i}-2)(q-1)}(\mathbb{F}_{q}^{n}) & \text{\text{if }}|B_{i}|=1,\\
\Poly_{(r_{i}-1)(q-1)}^{2}(\mathbb{F}_{q}^{n}) & \text{if }|B_{i}|>1.
\end{cases}
\]
Then we have that 

\begin{equation}
\prank(J_{k})\le\sum_{i=1}^{l}\dim\mathcal{V}_{i}.\label{eq:prank_bound_abstract_sets}
\end{equation}
\end{prop}
\begin{proof}
We will show that $J_{k}$ can be decomposed as a linear combination
of terms of the form 
\begin{equation}
\delta(\vec{x}_{B_{i}})Q(x_{b_{i}})T(\vec{x}_{\bar{B}_{i}})\label{eq:decomposition_into_pieces}
\end{equation}
 where $Q(x_{b_{i}})\in\mathcal{V}_{i}$ and $T(\vec{x}_{\bar{B}_{i}})$
is a function that involves only the variables with indices in $\bar{B}_{i}=\{1,\dots,k+1\}\backslash B_{i}$.
By the definition of the partition rank, this implies (\ref{eq:prank_bound_abstract_sets}).
Following equation (\ref{eq:RkA_term}), we can decompose $J_{k}$
as a linear combination of terms of the form
\[
R_{k+1}(x_{1},\dots,x_{k+1})\prod_{A\in P}\delta\left(\vec{x}_{A}\right).
\]
By definition, there exists some index $j$ such that $B_{j}\in P$.
Let $i$ be the minimal index such that $B_{i}\in P$. Then by definition
of $r_{i}$, $P$ must be a partition of $\{1,\dots,k+1\}$ into at
most $r_{i}$ parts. It follows from Lemma \ref{lem:R_k_partition_degree_bound}
that 
\[
R_{k+1}(x_{1},\dots,x_{k+1})\prod_{A\in P}\delta\left(\vec{x}_{A}\right)=\sum_{j}\delta\left(\vec{x}_{B_{i}}\right)Q_{j}(x_{b_{i}})T_{j}(\vec{x}_{\bar{B}_{i}})
\]
where for each $j$
\[
Q_{j}(x_{b_{i}})\in\mathcal{V}_{i}.
\]
Note that the product $\prod_{A\in P,A\neq B_{i}}\delta\left(\vec{x}_{A}\right)$
was not dropped, but rather is contained in the $T_{j}$ term. Thus
it follows that $J_{k}$ is a linear combination of terms of the form
in (\ref{eq:decomposition_into_pieces}), and we obtain equation (\ref{eq:prank_bound_abstract_sets}).
\end{proof}
We now prove Proposition \ref{prop:J_k_prank_bound}, which completes
the proof of Theorem \ref{thm:Right-Corner} by Lemma \ref{lem:thm_from_J_k_prank_bound}.
\begin{proof}[\emph{\emph{Proof of Proposition \ref{prop:J_k_prank_bound}}} ]
Let the sequence of subsets $B_{1},\dots,B_{2^{k}-1}$ be defined
by listing the non-empty subsets of $\{1,\dots,k\}$ in order by their
cardinality, with ties broken by lexicographical order of the elements
of the set. That is, 
\[
B_{1}=\{1\},B_{2}=\{2\},\dots,B_{k}=\{k\},B_{k+1}=\{1,2\},\dots,B_{2^{k}-1}=\{1,\dots,k\}.
\]
Let 
\[
d_{i}=\begin{cases}
(r_{i}-2)(q-1) & \text{if }|B_{i}|=1,\\
(r_{i}-1)(q-1) & \text{if }|B_{i}|>1,
\end{cases}
\]
where $r_{i}$ was defined in (\ref{eq:partition_max_criteria}).
By Proposition \ref{prop:sequence_of_sets_bound} we have the upper
\begin{equation}
\prank(J_{k})\leq\sum_{i:\ |B_{i}|=1}\dim\Poly_{d_{i}}\left(\mathbb{F}_{q}^{n}\right)+\sum_{i:\ |B_{i}|>1}\dim\Poly_{d_{i}}^{2}\left(\mathbb{F}_{q}^{n}\right),\label{eq:prank_bound_explicit}
\end{equation}
and so we will obtain an upper bound for $\prank(J_{k})$ by bounding
each of $r_{1},\dots r_{2^{k}-1}$. We begin by bounding the $r_{i}$
corresponding to the $\binom{k}{1}$ sets of size $1$. For $i=1$,
the maximal partition into subsets of size $1$ satisfies the criteria,
and so 
\[
r_{1}=k+1.
\]
For each $i\in\{2,\dots,k\}$, the maximal partition can have at most
$k+1-(i-1)$ singletons, since the singletons $\{1,\dots,i-1\}$ are
already covered by $B_{1},\dots,B_{i-1}$. These first $i-1$ elements
in $\{1,\dots,k+1\}$ can form at most $\lfloor\frac{i-1}{2}\rfloor$
pairs. For integers $i$ we have that $-(i-1)+\lfloor\frac{i-1}{2}\rfloor=-\lfloor\frac{i}{2}\rfloor$,
and so we have the upper bound 
\[
r_{i}\leq k+1=k+1-\biggr\lfloor\frac{i}{2}\biggr\rfloor.
\]
For $2\leq i\leq k$, $r_{i}$ this will be bounded above by $k-1$,
and so by Lemma \ref{lem:balls_urns} 
\begin{equation}
\sum_{i:\ |B_{i}|=1}\dim\Poly_{d_{i}}\left(\mathbb{F}_{q}^{n}\right)\leq\binom{n+(k-1)(q-1)}{(k-1)(q-1)}+(k-1)\binom{n+(k-2)(q-1)}{(k-2)(q-1)}.\label{eq:B_i_equal_1_bound}
\end{equation}
What remains is to bound $\sum_{i:\ |B_{i}|>1}\dim\Poly_{d_{i}}^{2}\left(\mathbb{F}_{q}^{n}\right)$,
and to combine these bounds in a simple manner. For the $\binom{k}{j}$
subsets of size $j>1$, note that any partition cannot contain any
sets of size less than $j$ that do not contain $k+1$, since those
partitions have already been accounted for. Thus for each of these
$\binom{k}{j}$ sets $B_{i}$ with $|B_{i}|=j$, we have
\[
r_{i}\leq1+\biggr\lfloor\frac{k}{j}\biggr\rfloor.
\]
By Lemma \ref{lem:balls_urns} 
\begin{equation}
\sum_{i:\ |B_{i}|>1}\dim\Poly_{d_{i}}^{2}\left(\mathbb{F}_{q}^{n}\right)\leq\sum_{j=2}^{k}\binom{k}{j}\binom{n+1+\lfloor\frac{k}{j}\rfloor(q-1)}{\lfloor\frac{k}{j}\rfloor(q-1)}.\label{eq:B_i_above_1_bound}
\end{equation}
To combine these two bounds, we make use of the Chu-Vandermonde identity,
which states that for any non-negative integers $a,b,c$
\[
\binom{a+b}{c}=\sum_{j=0}^{c}\binom{a}{j}\binom{b}{c-j}.
\]
Letting $a=k,$ $b=n+(k-1)(q-1)-1$, and $c=(k-1)(q-1)$, we have
that 
\begin{equation}
\binom{n+(k-1)q}{(k-1)(q-1)}=\sum_{j=0}^{k}\binom{k}{j}\binom{n+(k-1)(q-1)-1}{(k-1)(q-1)-j}\label{eq:chu_vandermonde}
\end{equation}
where the sum ends at $k$ since $(k-1)(q-1)\geq k$. The sum of the
$j=0$ and $j=1$ terms above can be written as 
\[
\binom{n+(k-1)(q-1)}{(k-1)(q-1)}+(k-1)\binom{n+(k-1)(q-1)-1}{(k-1)(q-1)-1},
\]
which is an upper bound for $\sum_{i:\ |B_{i}|=1}\dim\Poly_{d_{i}}\left(\mathbb{F}_{q}^{n}\right)$
by (\ref{eq:B_i_equal_1_bound}). For $k\geq3$ and $j\geq2$, we
have the bound 
\begin{equation}
\biggr\lfloor\frac{k}{j}\biggr\rfloor(q-1)\leq(k-1)(q-1)-j\label{eq:helper_simple_bound}
\end{equation}
For $2\leq j\leq k-1$, this follows from the fact that $\lfloor\frac{k}{j}\rfloor\leq k-j$,
and for $k=j$, it follows since $q>k$ by assumption. Thus for every
$2\le j\leq k$
\[
\binom{n+1+\lfloor\frac{k}{j}\rfloor(q-1)}{\lfloor\frac{k}{j}\rfloor(q-1)}\leq\binom{n+1+(k-1)(q-1)-j}{(k-1)(q-1)-j}.
\]
Consequently, by (\ref{eq:B_i_equal_1_bound}), (\ref{eq:B_i_above_1_bound}),
and the identity (\ref{eq:chu_vandermonde}), it follows that 
\[
\prank(J_{k})\leq\sum_{i}\dim\mathcal{V}_{i}\leq\binom{n+(k-1)q}{(k-1)(q-1)}.\qedhere
\]
\end{proof}

\specialsection*{Acknowledgments}

I would like to thank Will Sawin for his helpful comments and for
his simple proofs of lemmas \ref{lem:Group-action-lemma} and \ref{lem:critical-indicator-function}.
I would also like to thank Lisa Sauermann for her many helpful comments,
corrections, and suggestions. I am grateful to Avi Wigderson for referring
me to \cite{ShpilkaYehudayoffCircuitsSurvey}. This work was partially
supported by the NSERC PGS-D scholarship, and by Ben Green's ERC Starting
Grant 279438, Approximate Algebraic Structure and Applications.

\bibliographystyle{plain}

\end{document}